\documentclass[fullpage, 11pt]{article}

\usepackage[english]{babel}
\usepackage{amsfonts}
\usepackage{amssymb}
\usepackage{amsmath}
\usepackage{latexsym}
\usepackage{amsthm}
\usepackage{epsfig}
\usepackage{graphicx}
\usepackage{color}
\usepackage{subfigure}

\newcommand{\bb}{\mathbb}
\newcommand{\conv}{\mathrm{conv}}

\newcommand{\lin}{\mathrm{lin}}
\newcommand{\R}{\bb R}
\newcommand{\Q}{\bb Q}
\newcommand{\Z}{\bb Z}

\newcommand{\intr}{\mathrm{\bf int}}

\newcommand{\cl}{\mathrm{\bf cl}}
\newcommand{\bd}{\mathrm{\bf bd}}

\newcommand{\rec}{\mathrm{rec}}

\newcommand{\old}[1]{{}}

\newtheorem{prop}{Proposition}
\newtheorem{theorem}[prop]{Theorem}
\newtheorem{corollary}[prop]{Corollary}
\newtheorem{lemma}[prop]{Lemma}

\newtheorem{claim}{Claim}

\newtheorem{remark}[prop]{Remark}

\newtheorem{example}[prop]{Example}

\newenvironment{pf}{\begin{trivlist} \item[] {\em Proof:}}{\hfill $
\Box$
                       \end{trivlist}}
\newcommand{\sm}{\setminus}

\def\st{\,|\,}

\old{ \addtolength{\oddsidemargin}{-30pt}
\addtolength{\evensidemargin}{-30pt} \addtolength{\textwidth}{60pt}

\addtolength{\topmargin}{-22pt} \addtolength{\textheight}{32pt} }

\begin{document}
\title{Minimal inequalities for an infinite\\ relaxation of integer programs}

\author{%
Amitabh Basu \\
Carnegie Mellon University,
abasu1@andrew.cmu.edu\\\\
Michele Conforti \\
Universit\`a di Padova,
conforti@math.unipd.it \\ \\
G\'erard Cornu\'ejols \thanks{
Supported by   NSF grant CMMI0653419,
ONR grant N00014-03-1-0188 and ANR grant BLAN06-1-138894.} \\
Carnegie Mellon University and Universit\'e d'Aix-Marseille \\
gc0v@andrew.cmu.edu \\ \\
Giacomo Zambelli \\
Universit\`a di Padova,
giacomo@math.unipd.it
}

\date{April 17, 2009, revised November 19, 2009}

\maketitle

\begin{abstract}
We show that maximal $S$-free convex sets are polyhedra when $S$ is
the set of integral points in some rational polyhedron of $\R^n$.
This result extends a theorem of Lov\'asz characterizing maximal
lattice-free convex sets. Our theorem has implications in integer
programming. In particular, we show that maximal $S$-free convex
sets are in one-to-one correspondence with minimal inequalities.
\end{abstract}

\maketitle

\section{Introduction}

Consider a mixed integer linear program, and the optimal tableau of the linear programming
relaxation. We select $n$ rows of the tableau, relative to $n$ basic integer variables
$x_1,\ldots, x_n$.  Let $s_1,\ldots, s_m$ denote the nonbasic variables. Let $f_i\geq 0$ be the value of $x_i$ in the basic solution associated with the tableau, $i=1,\ldots,n$, and suppose $f\notin\Z^n$. The tableau
restricted to these $n$ rows is of the form
\begin{equation}\label{eq:tableau}x=f+\sum_{j=1}^m r^js_j,\quad x\geq 0\mbox{ integral},\, s\geq 0,\, \mbox { and } s_j\in \Z, j\in I,\end{equation}
where $r^j\in\R^n$, $j=1,\ldots, m$, and $I$ denotes the set of
integer nonbasic variables.

An important question in integer programming is to derive valid inequalities for~\eqref{eq:tableau}, cutting off the current infeasible solution $x=f$, $s=0$. We will consider a simplified model where the integrality conditions are relaxed on all nonbasic variables. On the other hand, we can present our results in a more general context, where the constraints $x\geq 0$, $x\in\Z^n$, are replaced by constraints $x\in S$, where $S$ is the set of integral points in some given rational polyhedron such that $\dim(S)=n$, i.e. $S$ contains $n+1$ affinely independent points. Recall that a polyhedron $Ax \leq b$ is {\em rational} if the matrix $A$ and vector $b$ have rational entries.

So we study the following model, introduced by Johnson~\cite{Johnson}.
\begin{equation}\label{eq:finite}x=f+\sum_{j=1}^m r^js_j,\quad x\in S,\, s\geq 0,\end{equation}
where $f\in\conv(S)\setminus\Z^n$.
Note that every inequality cutting off the point $(f,0)$ can be expressed in terms of the nonbasic variables $s$ only, and can therefore be written in the form $\sum_{j=1}^m \alpha_j s_j\geq 1$.

In this paper we are interested in ``formulas'' for deriving such inequalities. More formally, we are interested in functions $\psi\,:\,\R^n\rightarrow \R$ such that the inequality
$$\sum_{j=1}^m \psi(r^j) s_j\geq 1$$
is valid for~\eqref{eq:finite} for every choice of $m$ and vectors $r^1,\ldots,r^m\in\R^n$. We refer to such functions $\psi$ as {\em valid functions} (with respect to $f$ and $S$). Note that, if $\psi$ is a valid function and $\psi'$ is a function such that $\psi\leq\psi'$, then $\psi' $ is also valid, and the inequality $\sum_{j=1}^m \psi'(r^j) s_j\geq 1$ is implied by $\sum_{j=1}^m \psi(r^j) s_j\geq 1$. Therefore we only need to investigate (pointwise) minimal valid functions.

Andersen, Louveaux, Weismantel, Wolsey~\cite{AndLouWeiWol} characterize minimal valid functions for the case $n=2$, $S=\Z^2$. Borozan and Cornu\'ejols~\cite{BorCor} extend this result to $S=\Z^n$ for any $n$. These papers and a result of Zambelli~\cite{giacomo} show a one-to-one correspondence between minimal valid functions and maximal lattice-free convex sets with $f$ in the interior. These results have been further generalized in~\cite{BaCoCoZa}. Minimal valid functions for the case $S=\Z^n$ are intersection cuts~\cite{bal1}.

Our interest in model~\eqref{eq:finite} arose from a recent paper of Dey and Wolsey~\cite{DW}. They introduce the notion of {\em $S$-free
convex set} as a convex set without points of $S$ in its interior, and show the connection between valid functions and $S$-free convex sets with $f$ in their interior.

A class of valid functions can be defined as follows. A function $\psi$ is {\em positively homogeneous} if  $\psi(\lambda r)=\lambda(\psi r)$ for every $r\in\R^n$ and every $\lambda\geq 0$, and it is {\em subadditive} if $\psi(r)+\psi(r')\geq \psi(r+r')$ for all $r,r'\in\R^n$. A function $\psi$ is {\em sublinear} if it is positively homogeneous and subadditive. It is easy to observe that sublinear functions are also convex.

Assume that $\psi$ is a sublinear function such that the set \begin{equation}\label{eq:B_psi} B_\psi=\{x\in\R^n\st \psi(x-f)\leq 1\}\end{equation}
is $S$-free. Note that $B_\psi$ is closed and convex because $\psi$ is convex. Since $\psi$ is positively homogeneous, $\psi(0)=0$, thus $f$ is in the interior of $B_\psi$. We claim that $\psi$ is a valid function. Indeed, given any solution $(\bar x,\bar s)$ to~\eqref{eq:finite}, we have
$$\sum_{j=1}^m \psi(r^j)\bar s_j\geq \psi(\sum_{j=1}^m r^j\bar s_j)=\psi(\bar x-f)\geq 1,$$
where the first inequality follows from sublinearity and the last one follows from the fact that $\bar x$ is not in the interior of $B_\psi$.

Dey and Wolsey~\cite{DW} show that every minimal valid function $\psi$ is sublinear and $B_\psi$ is an $S$-free convex set with $f$ in its interior. In this paper, we prove that if $\psi$ is a minimal valid function, then $B_\psi$ is a {\em maximal $S$-free convex set}.

In Section~\ref{sec:S-free}, we show that maximal $S$-free convex sets are polyhedra.
Therefore a maximal $S$-free convex set $B\subseteq \R^n$ containing $f$ in its interior can be uniquely written in the form $B=\{x\in \R^n:\;a_i(x-f)\le 1,\;  i=1,\ldots,k\}$.
Let $\psi_B : \R^n \rightarrow \R$ be the  function defined by
\begin{equation}\label{eq:psiB}
\psi_B(r)= \max_{i=1,\ldots,k} a_ir, \quad \forall r\in \R^n.
\end{equation}

It is easy to observe that the above function is sublinear and $B=\{x\in\R^n\st \psi_B(x-f)\leq 1\}$. In Section~\ref{sec:Rf} we will prove that every minimal valid function is of the form $\psi_B$ for some maximal $S$-free convex set $B$ containing $f$ in its interior.
Conversely, if $B$ is a maximal
$S$-free convex set containing
$f$ in its interior, then $\psi_B$ is a minimal valid function.

\section{Maximal $S$-free convex sets}

\label{sec:S-free}

Let $S\subseteq \Z^n$ be the set of integral points in some rational
polyhedron of $\R^n$. We say that $B\subset\R^n$ is an {\em $S$-free
convex set} if $B$ is convex and does not contain any point of $S$
in its interior. We say that $B$ is a {\em maximal $S$-free convex
set} if it is an $S$-free convex set and it is not properly
contained in any $S$-free convex set. It follows from Zorn's lemma that every $S$-free convex set is contained in a maximal $S$-free convex set.

When $S=\Z^n$, an $S$-free convex set is called a {\em lattice-free
convex set}. The following theorem of Lov\'asz characterizes maximal
lattice-free convex sets. A linear subspace or cone in $\R^n$ is
{\em rational} if it can be generated by rational vectors, i.e.
vectors with rational coordinates.

\begin{theorem}\label{thm:lattice-free} {\em (Lov\'asz \cite{Lovasz})} A set $B\subset \R^n$ is a maximal lattice-free convex set if and only if one of the following holds:
\begin{itemize}
\item[(i)] $B$ is a polyhedron of the form $B= P+L$ where $P$ is a polytope, $L$ is a rational linear space, $\dim(B)=\dim(P)+\dim(L)=n$, $B$ does not contain any integral point in its interior and there is an integral point in the relative interior of each facet of $B$;
\item[(ii)] $B$ is a hyperplane of $\R^n$ that is not rational.
\end{itemize}
\end{theorem}

Lov\'asz only gives a sketch of the proof. A complete proof can be
found in \cite{BaCoCoZa}. The next theorem is
an extension of Lov\'asz' theorem to maximal $S$-free convex sets.

Given a convex set $K\subset \R^n$, we denote by $\rec(K)$ its
recession cone and by $\lin(K)$ its lineality space. Given a set
$X\subseteq \R^n$, we denote by $\langle X\rangle$ the linear space
generated by $X$.
Given a $k$-dimensional linear space $V$ and a subset $\Lambda$ of
$V$, we say that $\Lambda$ is a {\em lattice of $V$} if there exists
a linear bijection $f:\,\R^k\rightarrow V$ such that
$\Lambda=f(\Z^k)$.

\begin{theorem}\label{thm:S-free} Let $S$ be the set of integral points in some rational polyhedron of $\R^n$ such that $\dim(S)=n$. A set $B\subset \R^n$ is a maximal $S$-free convex set if and only if one of the following holds:
\begin{itemize}
\item[$(i)$] $B$ is a polyhedron such that $B\cap \conv(S)$ has nonempty interior,  $B$ does not contain any point of $S$ in its interior and there is a point of $S$ in the relative interior of each of its facets.
\item[$(ii)$] $B$ is a half-space of $\R^n$ such that $B\cap \conv(S)$ has empty interior and the boundary of $B$ is a supporting hyperplane of $\conv(S)$.
\item[$(iii)$] $B$ is a hyperplane of $\R^n$ such that $\lin(B)\cap\rec(\conv(S))$ is not rational.
\end{itemize}
Furthermore, if (i) holds, the recession cone of $B\cap\conv(S)$ is rational and it is contained in the lineality space of $B$.
\end{theorem}

We illustrate case (i) of the theorem in the plane in Figure~\ref{fig:S-free}.
The question of the polyhedrality of maximal $S$-free convex sets was raised by Dey and Wolsey~\cite{DW}. They proved that this is the case for a maximal $S$-free convex set $B$, under the assumptions that $B\cap \conv(S)$ has nonempty interior and that
the recession cone of $B\cap\conv(S)$ is finitely generated and rational. Theorem~\ref{thm:S-free} settles the question in general.
\begin{figure}[htbp]\label{fig:S-free}
\centering
 \includegraphics[height=2.75in]{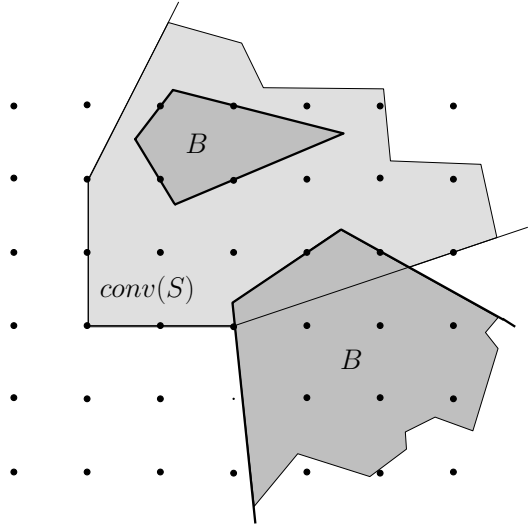}
\caption{Two examples of $S$-free sets in the plane (case (i) of Theorem~\ref{thm:S-free}). The light gray region indicates $\conv(S)$
and the dark grey regions illustrate the $S$-free sets.
A jagged line indicates that the region extends to infinity.}
\end{figure}

To prove Theorem~\ref{thm:S-free} we will need the following lemmas. The first one is proved in~\cite{BaCoCoZa} and is an easy consequence of Dirichlet's theorem.

\begin{lemma}\label{lemma:half-line}
 Let $y\in \Z^n$ and $r\in \R^n$. For every $\varepsilon >0$ and $\bar\lambda\geq 0$, there exists an integral point at distance less than $\varepsilon$ from the half line
$\{y+\lambda r\st \lambda\geq \bar\lambda\}.$
\end{lemma}

\begin{lemma}\label{lemma:lineality} Let $B$ be an $S$-free convex set such that $B\cap \conv(S)$ has nonempty  interior. For every $r\in\rec(B)\cap\rec(\conv(S))$, $B+\langle r\rangle$ is $S$-free.
\end{lemma}
\begin{proof}
Let $C=\rec(B)\cap\rec(\conv(S))$ and  $r\in C\sm\{0\}$. Suppose by contradiction that there exists $y\in S\cap \intr(B+\langle r\rangle)$. We show that $y \in \intr(B) + \langle r\rangle$. If not, $(y + \langle r\rangle) \cap \intr(B) = \emptyset$, which implies that there is a hyperplane $H$ separating the line $y + \langle r\rangle$ and $B + \langle r\rangle$, a contradiction. Thus there exists  $\bar\lambda$ such that  $\bar y=y+\bar\lambda r\in \intr(B)$, i.e. there exists $\varepsilon>0$ such that $B$ contains the open ball $B_\varepsilon(\bar y)$ of radius $\varepsilon$ centered at $\bar y$.
Since $r\in C \subseteq \rec(B)$, it follows that  $B_\varepsilon(\bar y)+\{\lambda r\st \lambda\geq 0\}\subset B$. Since $y\in\Z^n$, by Lemma~\ref{lemma:half-line} there exists $z\in \Z^n$ at distance less than $\varepsilon$ from the half line $\{y+\lambda r\st \lambda\geq \bar\lambda\}$. Thus $z\in B_\varepsilon(\bar y)+\{\lambda r\st \lambda\geq 0\}$, hence $z\in\intr(B)$. Note that the half-line $\{y+\lambda r\st \lambda\geq \bar\lambda\}$ is in $\conv(S)$, since $y\in S$ and $r\in\rec(\conv(S))$.
Since $\conv(S)$ is a rational polyhedron, for $\varepsilon>0$ sufficiently small every integral point at distance at most $\varepsilon$ from $\conv(S)$ is in $\conv(S)$. Therefore $z\in S$, a contradiction.
\end{proof}

\begin{proof}[Proof of Theorem~\ref{thm:S-free}] The proof of the ``if'' part is standard, and it is similar to the proof for the lattice-free case (see~\cite{BaCoCoZa}). We show the ``only if'' part. Let $B$ be a maximal $S$-free convex set. If $\dim(B)<n$, then $B$ is contained in some affine hyperplane $K$. Since $K$ has empty interior, $K$ is $S$-free, thus $B=K$ by maximality of $B$.
Next we show that $\lin(B)\cap\rec(\conv(S))$ is not rational.
Suppose not. Then the linear subspace
$L=\langle\lin(B)\cap\rec(\conv(S))\rangle$ is rational. Therefore
the projection $\Lambda$ of $\mathbb{Z}^n$ onto $L^\bot$ is a
lattice of $L^\bot$ (see, for example, Barvinok \cite{barv} p 284
problem 3). The projection $S'$ of  $S$ onto $L^\bot$ is a subset of
$\Lambda$. Let $B'$ be the projection of $B$ onto $L^\bot$. Then
$B'\cap\conv(S')$ is the projection of $B\cap\conv(S)$ onto
$L^\bot$. Since $B$ is a hyperplane, $\lin(B) = \rec(B)$. This
implies that $B'\cap\conv(S')$ is bounded : otherwise there is an
unbounded direction $d \in L^\bot$ in $\rec(B')\cap\rec(\conv(S'))$
and so $d + l \in \rec(B)\cap\rec(\conv(S))$ for some $l\in L$.
Since $\rec(B)\cap\rec(\conv(S))=\lin(B)\cap\rec(\conv(S))$, this
would imply that $d \in L$ which is a contradiction. Fix $\delta
> 0$. Since $\Lambda$ is a lattice and $S'\subseteq \Lambda$, there
is a finite number of points at distance less than $\delta$ from the
bounded set $B'\cap\conv(S')$ in $L^\bot$. It follows that there
exists $\varepsilon>0$ such that every point of $S'$ has distance at
least $\varepsilon$ from  $B'\cap\conv(S')$. Let $B''=\{v+w\st  v\in
B,\, w\in L^\bot, \|w\|\leq\varepsilon\}$. The set $B''$ is $S$-free
by the choice of  $\varepsilon$, but $B''$ strictly contains $B$,
contradicting the maximality of $B$. Therefore $(iii)$ holds when
$\dim(B)<n$. Hence we may assume $\dim(B)=n$. If $B\cap \conv(S)$
has empty interior, then there exists a hyperplane separating $B$
and $\conv(S)$ which is supporting for $\conv(S)$. By maximality of
$B$ case $(ii)$ follows.

Therefore we may assume that  $B\cap \conv(S)$ has nonempty interior. We show that $B$ satisfies $(i)$.

\begin{claim} \label{claim:P} There exists a rational polyhedron $P$ such that:\\
$i)$ $\conv(S)\subset\intr(P)$,\\
$ii)$ The set $K=B\cap P$ is lattice-free,\\
$iii)$ For every facet $F$ of $P$, $F\cap K$ is a facet of $K$,\\
$iv)$ For every facet $F$ of $P$, $F\cap K$ contains an integral point in its relative interior.
\end{claim}

Since  $\conv(S)$ is a rational polyhedron, there exist integral $A$
and $b$ such that $\conv(S)=\{x\in\R^n\st Ax\leq b\}$. The set
$P'=\{x\in\R^n\st Ax\leq b+\frac 12 \mathbf{1}\}$ satisfies $i)$.
The set $B\cap P'$ is lattice-free since $B$ is $S$-free and $P'$
does not contain any point in $\Z^n\sm S$, thus $P'$ also satisfies
$ii)$. Let $\bar Ax\leq \bar b$ be the system containing all
inequalities of $Ax\leq b+\frac 12 \mathbf{1}$ that define facets of
$B\cap P'$. Let $P_0=\{x\in\R^n\st \bar Ax\leq \bar b\}$. Then $P_0$
satisfies $i), ii),iii)$. See Figure~\ref{fig:proof} for an
illustration.

\begin{figure}[htbp]\label{fig:proof}
\centering
\includegraphics[scale=.37]{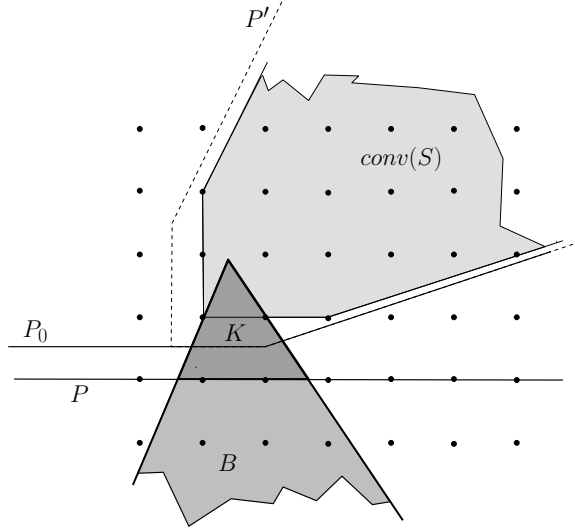}
\caption{Illustration for Claim~\ref{claim:P}.}
\end{figure}

It will be more convenient to write $P_0$ as intersection of the
half-spaces defining the facets of $P_0$, $P_0=\cap_{H\in
\mathcal{F}_0} H$. We construct a sequence of rational polyhedra
$P_0\subset P_1\subset\ldots\subset P_t$ such that $P_i$ satisfies
$i),ii),iii)$, $i=1,\ldots,t$, and such that $P_t$ satisfies $iv)$.
Given $P_i$, we construct $P_{i+1}$ as follows. Let  $P_i=\cap_{H\in
\mathcal{F}_i} H$, where $\mathcal F_i$ is the set of half spaces
defining facets of $P_i$. Let $\bar H$ be a half-space in
$\mathcal{F}_i$ defining a facet of $B\cap P_i$ that does not
contain an integral point in its relative interior; if no such $\bar
H$ exists, then $P_i$ satisfies $iv)$ and we are done. If
$B\cap\bigcap_{H\in \mathcal{F}_i\sm\{\bar H\}}H$ does not contain
any integral point in its interior, let
$\mathcal{F}_{i+1}=\mathcal{F}_i\sm\{\bar H\}$. Otherwise, since
$P_i$ is rational, among all  integral points in the interior of
$B\cap\bigcap_{H\in \mathcal{F}_i\sm\{\bar H\}}H$ there exists one,
say $\bar x$, at minimum distance from $\bar H$. Let $H'$ be the
half-space containing $\bar H$ with $\bar x$ on its boundary. Let
$\mathcal F_{i+1}=\mathcal{F}_i\sm\{\bar H\}\cap\{H'\}$. Observe
that $H'$ defines a facet of $P_{i+1}$ since $\bar x$ is in the
interior of $B\cap\bigcap_{H\in \mathcal{F}_{i+1}\sm\{H'\}}H$ and it
is on the boundary of $H'$. So $i),ii),iii)$ are satisfied and
$P_{i+1}$ has fewer facets that violate $iv)$ than $P_i$.\hfill
$\diamond$
\bigskip

Let $T$ be a maximal lattice-free convex set containing the set $K$ defined in Claim~\ref{claim:P}. As remarked earlier, such a set $T$ exists. By Theorem~\ref{thm:lattice-free}, $T$ is a polyhedron with an integral point in the relative interior of each of its facets. Let $H$ be a hyperplane that defines a facet of $P$. Since $K\cap H$ is a facet of $K$ with an integral point in its relative interior, it follows that $H$ defines a facet of $T$. This implies that $T\subset P$. Therefore we can write $T$ as
\begin{equation}\label{eq:T}T=P\cap\bigcap_{i=1}^k H_i,\end{equation} where $H_i$ are halfspaces. Let $\bar H_i=\R^n\sm\intr(H_i)$, $i=1,\ldots, k$.

\begin{claim}\label{claim:B-poly} $B$ is a polyhedron.
\end{claim}

We first show that, for $i=1,\ldots,k$, $\intr(B)\cap(\bar H_i\cap \conv(S))=\emptyset$. Consider $y\in\intr(B)\cap\bar H_i$. Since $y\in\bar H_i$ and $K$ is contained in $T$, $y\notin\intr(K)$. Since $K=B\cap P$ and $y\in \intr(B)\sm\intr(K)$, it follows that $y\notin \intr(P)$. Hence $y\notin \conv(S)$ because $\conv(S)\subseteq\intr(P)$.

Thus, for $i=1,\ldots,k$, there exists a hyperplane separating $B$ and $\bar H_i\cap \conv(S)$. Hence there exists a halfspace $K_i$ such that $B\subset K_i$ and $\bar H_i\cap \conv(S)$ is disjoint from the interior of $K_i$. We claim that the set $B'=\cap_{i=1}^k K_i$ is $S$-free. Indeed, let $y\in S$. Then $y$ is not interior of $T$. Since $y\in\conv(S)$ and $\conv(S)\subseteq\intr(P)$, $y$ is in the interior of $P$. Hence, by \eqref{eq:T}, there exists $i\in\{1,\ldots,k\}$ such that $y$ is not in the interior of $H_i$. Thus $y\in \bar H_i\cap \conv(S)$. By construction, $y$ is not in the interior of $K_i$, hence $y$ is not in the interior of $B'$. Thus $B'$ is an $S$-free convex set containing $B$. Since $B$ is maximal, $B'=B$.\hfill $\diamond$

\begin{claim}\label{claim:lin-K} $\lin(K)=\rec(K)$.
\end{claim}
Let $r\in \rec(K)$. We show $-r\in\rec(K)$. By Lemma~\ref{lemma:lineality} applied to $\Z^n$, $K+\langle r\rangle$ is lattice-free. We observe that $B+\langle r\rangle$ is $S$-free. If not, let $y\in S\cap\intr(B+\langle r\rangle)$. Since $S\subseteq \intr(P)$, $y\in\intr(P+\langle r\rangle)$, hence $y\in \intr(K+\langle r\rangle)$, a contradiction. Hence, by maximality of $B$, $B=B+\langle r\rangle$. Thus $-r\in\rec(B)$. Suppose that $-r\notin\rec(P)$. Then there exists a facet $F$ of $P$ that is not parallel to $r$. By construction, $F\cap K$ is a facet of $K$ containing an integral point $\bar x$ in its relative interior. The point $\bar x$ is then in the interior of $K+\langle r\rangle$, a contradiction.\hfill $\diamond$

\begin{claim}\label{claim:lin-K-rat} $\lin(K)$ is rational.
\end{claim}

Consider the maximal lattice-free convex set $T$ containing $K$ considered earlier. By Theorem~\ref{thm:lattice-free}, $\lin(T)=\rec(T)$, and $\lin(T)$ is rational. Clearly $\lin(T)\supseteq \lin(K)$. Hence, if the claim does not hold, there exists a rational vector $r\in\lin(T)\sm \lin(K)$. By~\eqref{eq:T}, $r\in \lin(P)$. Since $K=B\cap P$, $r\notin\lin(B)$. Hence $B\subset B+\langle r\rangle$. We will show that $B+\langle r\rangle$ is $S$-free, contradicting the maximality of $B$. Suppose there exists $y\in S\cap\intr(B+\langle r\rangle)$. Since $\conv(S)\subseteq\intr(P)$, $y\in\intr(P)\subseteq \intr(P)+\langle r\rangle$. Therefore $y\in\intr(B\cap P)+\langle r\rangle$. Since $B\cap P\subseteq T$, then $y\in\intr(T)+\langle r\rangle=\intr(T)$ where the last equality follows from $r\in \lin(T)$. This contradicts the fact that $T$ is lattice-free. \hfill $\diamond$
\bigskip

By Lemma~\ref{lemma:lineality} and by the maximality of $B$, $\lin(B)\cap\rec(\conv(S))=\rec(B)\cap\rec(\conv(S))$.

\begin{claim} \label{claim:rational-cone} $\lin(B)\cap\rec(\conv(S))$ is rational.
\end{claim}
Since $\lin(K)$ and $\rec(\conv(S))$ are both rational, we only need
to show $\lin(B)\cap\rec(\conv(S))=\lin(K)\cap\rec(\conv(S))$. The
``$\supseteq$'' direction follows from $B\supseteq K$. For the other
direction, note that, since $\conv(S)\subseteq P$, we have
$\lin(B)\cap \rec(\conv(S))\subseteq \lin(B)\cap \rec(P)=\lin(B\cap
P)=\lin(K)$, hence $\lin(B)\cap\rec(\conv(S))\subseteq
\lin(K)\cap\rec(\conv(S))$.\hfill $\diamond$

\begin{claim} \label{claim:facets} Every facet of $B$ contains a point of $S$ in its relative interior.
\end{claim}

Let $L$ be the linear space generated by
$\lin(B)\cap\rec(\conv(S))$. By Claim~\ref{claim:rational-cone}, $L$
is rational. Let $B'$, $S'$, $\Lambda$ be the projections of $B$,
$S$, $\Z^n$, respectively, onto $L^\bot$. Since $L$ is rational,
$\Lambda$ is a lattice of $L^\bot$ and $S'=\conv(S')\cap\Lambda$.
Also, $B'$ is a maximal $S'$-free convex set of $L^\bot$, since for
any $S'$-free set $D$ of $L^\bot$, $D+L$ is $S$-free. Note that
$\lin(B)\cap\rec(\conv(S))=\rec(B)\cap\rec(\conv(S))$ implies that
$B'\cap\conv(S')$ is bounded. Otherwise there is an unbounded
direction $d \in L^\bot$ in $\rec(B')\cap\rec(\conv(S'))$ and so $d
+ l \in \rec(B)\cap\rec(\conv(S))$ for some $l\in L$. Since
$\rec(B)\cap\rec(\conv(S)) = \lin(B)\cap\rec(\conv(S))$, this would
imply that $d \in L$ which is a contradiction. Let $B'=\{x\in
L^\bot\st \alpha_i x\leq \beta_i,\,i=1,\ldots,t\}$. Given
$\varepsilon>0$, let $\bar B=\{x\in L^\bot\st \alpha_i x\leq
\beta_i,\,i=1,\ldots,t-1,\, \alpha_t x\leq \beta_t+\varepsilon\}$.
The polyhedron $\conv(S')\cap \bar B$ is a polytope since it has the
same recession cone as $\conv(S')\cap B'$. The polytope
$\conv(S')\cap \bar B$ contains points of $S'$ in its interior by
the maximality of $B'$. Since $\Lambda$ is a lattice of $L^\bot$,
$\intr(\conv(S')\cap \bar B)$ has a finite number of points in $S'$,
hence there exists one minimizing $\alpha_t x$, say $z$. By
construction, the polyhedron $B''=\{x\in L^\bot\st \alpha_i x\leq
\beta_i,\,i=1,\ldots,t-1,\, \alpha_t x\leq \alpha_t z\}$ does not
contain any point of $S'$ in its interior and contains $B'$. By the
maximality of $B'$, $B'=B''$ hence $B'$ contains $z$ in its relative
interior, and $B$ contains a point of $S$ in its relative interior.
\end{proof}

\begin{corollary} For every maximal $S$-free convex set $B$ there exists a maximal lattice-free convex set $K$ such that, for every facet $F$ of $B$, $F\cap K$ is a facet of $K$.
\end{corollary}
\begin{proof} Let $K$ be defined as in Claim~\ref{claim:P} in the proof of Theorem~\ref{thm:S-free}. It follows from the proof that $K$ is a maximal lattice-free convex set with the desired properties.
\end{proof}

\section{Minimal valid functions}

\label{sec:Rf}

In this section we study minimal valid functions. We find it convenient to state our results in terms of
an infinite model introduced by Dey and Wolsey~\cite{DW}.

Throughout this section, $S\subseteq \Z^n$ is a set of integral points in some rational
polyhedron of $\R^n$ such that $\dim(S)=n$, and  $f$ is a point in $\conv(S)\setminus \Z^n$.
Let $R_{f,S}$ be the set of all infinite dimensional vectors $s=(s_r)_{r\in\R^n}$ such that
\begin{eqnarray}\label{eq:RfS}
&&f+\sum_{r\in \R^n} r s_r\in S\nonumber\\
&&s_r\geq 0,\quad r\in \R^n\\
&&s \mbox{ has finite support}\nonumber
\end{eqnarray}
where $s$ has {\em finite support} means that $s_r$ is zero for all but a finite number of $r\in\R^n$.

A function $\psi\,:\,\R^n\rightarrow \R$ is {\em valid} (with respect to $f$ and $S$) if the linear
inequality
\begin{equation}\label{eq:valid}\sum_{r\in \R^n}\psi(r)s_r\geq 1\end{equation}
is satisfied by every $s\in R_{f,S}$.
Note that this definition coincides with the one we gave in the introduction.

Given two functions $\psi, \psi' $ we say that $\psi'$ {\em dominates} $\psi$ if $\psi'(r)\leq\psi(r)$ for all $r\in\R^n$.
A valid function $\psi$ is {\em minimal} if there is no valid function $\psi'\neq\psi$ that dominates $\psi$.

\begin{theorem}\label{thm:min-ineq-intr}
For every minimal valid function $\psi$, there exists a maximal $S$-free convex set $B$ with $f$ in its interior such that  $\psi_B$ dominates $\psi$.
Furthermore, if $B$ is a maximal
$S$-free convex set containing
$f$ in its interior, then $\psi_B$ is a minimal valid function.
\end{theorem}

We will need the following lemma.

\begin{lemma}\label{lemma:sublinear} Every valid function is dominated by a sublinear valid function.
\end{lemma}
\begin{proof}[Sketch of proof.] Given a valid function $\psi$, define the following function $\bar \psi$.
For all $\bar r\in \R^n$, let $\bar \psi(\bar r)=\inf\{\sum_{r\in\R^n}\psi(r)s_r\st \sum_{r\in\R^n}r s_r=\bar r,\, s\geq 0 \mbox{ with finite support}\}$.  Following the
proof of Lemma~18 in~\cite{BaCoCoZa} one can show that $\bar \psi$ is a valid sublinear function that dominates $\psi$.
\end{proof}

Given a  valid sublinear function $\psi$, the set $B_{\psi}=\{x\in \R^n\st \psi(x-f)\leq 1\}$ is closed, convex, and contains $f$ in its interior. Since $\psi$ is a valid function,
$B_\psi$ is $S$-free. Indeed the interior of $B_\psi$ is
$\intr(B_{\psi})=\{x\in \R^n\,:\, \psi(x-f)< 1\}$. Its boundary is
$\bd(B_{\psi})=\{x\in \R^n\,:\, \psi(x-f)= 1\}$, and its recession
cone is $\rec(B_{\psi})=\{x\in \R^n\,:\, \psi(x-f)\leq 0\}$. \medskip

Before proving Theorem~\ref{thm:min-ineq-intr}, we need the following general theorem about sublinear functions. Let $K$ be a closed, convex set in $\R^n$ with the origin in its interior. The {\em polar} of $K$ is the set $K^*=\{y\in \R^n\st ry\leq 1 \mbox{ for all } r\in K\}$. Clearly $K^*$ is closed and convex, and since $0\in \intr(K)$, it is well known that $K^*$ is bounded. In particular, $K^*$ is a compact set. Also, since $0\in K$, $K^{**}=K$. Let\begin{equation}\label{eq:set-T} \hat K=\{y\in K^*\st \exists x\in K\mbox{ such that } xy=1\}.\end{equation} Note that $\hat K$ is contained in the relative boundary of $K^*$. Let $\rho_K:\,\R^n\rightarrow \R$ be defined by
\begin{equation}\label{eq:support} \rho_K(r)=\sup_{y\in \hat K}ry,\quad \mbox{for all } r\in \R^n.\end{equation}
It is easy to show that $\rho_K$ is sublinear.

\begin{theorem}[Basu et al.~\cite{BaCoZa}]\label{thm:sublinear}
Let $K\subset \R^n$ be a closed convex set containing the origin in its interior. Then $K=\{r\in \R^n\st \rho_K(r)\leq 1\}$. Furthermore, for every sublinear function $\sigma$ such that $K=\{r\st \sigma(r)\leq 1\}$, we have $\rho_K(r)\leq \sigma(r)$ for every $r\in \R^n$.
\end{theorem}

\begin{remark}\label{rmk:K-poly} Let $K\subset \R^n$ be a polyhedron containing the origin in its interior. Let $a_1,\ldots, a_t\in \R^n$ such that $K=\{r\in \R^n\st a_ir\leq 1,\, i=1,\ldots, t\}$. Then $\rho_K(r)=\max_{i=1,\ldots,t}a_i r$.
\end{remark}
\begin{proof} The polar of $K$ is $K^*=\conv\{0,a_1,\ldots,a_t\}$ (see Theorem 9.1 in Schrijver~\cite{sch}). Furthermore,  $\hat K$ is the union of all the facets of $K^*$ that do not contain the origin, therefore
$$\rho_{K}(r)=\sup_{y\in\hat K}yr=\max_{i=1,\ldots,t}a_ir$$ for all $r\in \R^n$.
\end{proof}

\begin{remark}\label{rmk:valid} Let $B$ be a closed $S$-free convex set in $\R^n$ with $f$ in its interior, and let $K=B-f$. Then $\rho_K$ is a valid function.\end{remark}
\begin{pf} Let $s\in R_{f,S}$. Then $x=f+\sum_{r\in \R^n}rs_r$ is in $S$, therefore $x\notin \intr(B)$ because $B$ is $S$-free. By Theorem~\ref{thm:sublinear}, $\rho_K(x-f)\geq 1$. Thus
$$1\leq \rho_K(\sum_{r\in \R^n}rs_r)\leq \sum_{r\in \R^n}\rho_K(r)s_r,$$
where the second inequality follows from the sublinearity of $\rho_K$.
\end{pf}

\begin{lemma}\label{lemma:max-S-free} Let $C$ be a closed $S$-free convex set containing $f$ in its interior, and let $K=C-f$. There exists a maximal $S$-free convex set $B=\{x\in\R^n\st a_i(x-f)\leq 1,\,\, i=1,\ldots,k\}$ such that $a_i\in\cl(\conv(\hat K))$ for $i=1,\ldots,k$.
\end{lemma}
\begin{proof}
Since $C$ is an $S$-free convex set, it is contained in some maximal $S$-free convex set $T$. The set $T$ satisfies one of the statements (i)-(iii) of Theorem~\ref{thm:S-free}. By assumption, $f\in \conv(S)$ and $f$ is in the interior of $C$. Since $\dim(S)=n$, $\conv(S)$ is a full dimensional polyhedron, thus $\intr(C\cap \conv(S))\neq\emptyset$. This implies that $\intr(T \cap \conv(S))\neq \emptyset$, hence case (i) applies.

Thus $T$ is a polyhedron and $\rec(T\cap \conv(S))=\lin(T)\cap \rec(\conv(S))$ is rational. Let us choose $T$ such that the dimension of $\lin(T)$ is largest possible.

Since $T$ is a polyhedron containing $f$ in its interior, there exists $D\in\R^{t\times q}$ and $b\in\R^t$ such that $b_i>0$, $i=1,\ldots,t$, and $T=\{x\in \R^n\st D(x-f)\leq b\}$. Without loss of generality, we may assume that $\sup_{x\in C} d_i (x-f)=1$ where $d_i$ denotes the $i$th row of $D$, $i=1,\ldots, t$. By our assumption, $\sup_{r\in K} d_i r=1$. Therefore  $d_i\in K^*$, since $d_i r\leq 1$ for all $r\in K$. Furthermore $d_i\in\cl(\hat K)$, since $\sup_{r\in K} d_i r=1$.

Let $P=\{x\in \R^n\st D(x-f)\leq e\}$. Note that $\lin(P)=\lin(T)$. By our choice of $T$, $P+\langle r\rangle$ is not $S$-free for any $r\in\rec(\conv(S))\sm\lin(P)$, otherwise $P$ would be contained in a maximal $S$-free convex set whose lineality space contains $\lin(T)+\langle r\rangle$, a contradiction.

Let $L=\langle\rec(P\cap \conv(S))\rangle$. Since $\lin(P)=\lin(T)$, $L$ is a rational space. Note that $L\subseteq\lin(P)$, implying that $d_i\in L^\bot$ for $i=1,\ldots,t$.
\medskip

We observe next that we may assume that $P\cap\conv(S)$ is bounded. Indeed, let $\bar P$, $\bar S$, $\Lambda$ be the projections onto $L^\bot$ of $P$, $S$, and $\Z^n$, respectively. Since $L$ is a rational space, $\Lambda$ is a lattice of $L^\bot$ and $\bar S=\conv(\bar S)\cap \Lambda$. Note that $\bar P\cap \conv(\bar S)$ is bounded, since $L\supseteq \rec(P\cap\conv(S))$. If we are given a maximal $\bar S$-free convex set $\bar B$ in $L^\bot$ such that $\bar B= \{x\in L^\bot\st a_i(x-f)\leq 1,\, i=1,\ldots,h\}$ and $a_i\in\conv\{d_1,\ldots,d_t\}$ for $i=1,\ldots,h$, then $B=\bar B+L$ is the set $B= \{x\in \R^n\st a_i(x-f)\leq 1,\, i=1,\ldots,h\}$. Since $\bar B$ contains a point of $\bar S$ in the relative interior of each of its facets, $B$ contains a point of $S$ in the relative interior of each of its facets, thus $B$ is a maximal $S$-free convex set.
\medskip

Thus we assume that $P\cap\conv(S)$ is bounded, so $\dim(L)=0$. If all facets of $P$ contain a point of $S$ in their relative interior, then $P$ is a maximal $S$-free convex set, thus the statement of the lemma holds. Otherwise we describe a procedure that replaces one of the inequalities defining a facet of $P$ without any point of $S$ in its relative interior with an inequality which is a convex combination of the inequalities of $D(x-f)\leq e$, such that the new polyhedron thus obtained is $S$-free and has one fewer facet without points of $S$ in its relative interior. More formally, suppose the facet of $P$ defined by $d_1 (x-f)\leq 1$ does not contain any point of $S$ in its relative interior. Given $\lambda\in [0,1]$, let
$$P(\lambda)=\{x\in\R^n\st [\lambda d_1+(1-\lambda)d_2](x-f)\leq 1,\quad d_i(x-f)\leq 1\;\; i=2,\ldots, t\}.$$

Note that $P(1)=P$ and $P(0)$ is obtained from $P$ by removing the inequality $d_1(x-f)\leq 1$. Furthermore, given $0\leq \lambda'\leq\lambda''\leq 1$, we have $P(\lambda')\supseteq P(\lambda'')$.

Let $r_1,\ldots, r_m$ be generators of $\rec(\conv(S))$. Note that, since $P\cap\conv(S)$ is bounded, for every $j=1,\ldots,m$ there exists $i\in\{1,\ldots,t\}$ such that $d_ir_j>0$. Let $r_1,\ldots,r_h$ be the generators of $\rec(\conv(S))$ satisfying
\begin{eqnarray*}
d_1 r_j >0 &&\\
d_i r_j \leq 0 && i=2,\ldots, t.
\end{eqnarray*}
Note that, if no such generators exist, then $P(0)\cap \conv(S)$ is bounded. Otherwise  $P(\lambda)\cap \conv(S)$ is bounded if and only if, for $j=1,\ldots,h$
$$[\lambda d_1+(1-\lambda)d_2]r_j>0.$$
This is the case if and only if $\lambda>\lambda^*$, where $$\lambda^*=\max_{j=1,\ldots,h} \frac{-d_2r_j}{(d_1-d_2)r_j}.$$
Let $r^*$ be one of the vectors $r_1,\ldots,r_h$ attaining the maximum in the previous equation. Then $r^*\in\rec(P(\lambda^*\cap\conv(S)$.

Note that $P(\lambda^*)$ is not $S$-free otherwise $P(\lambda^*)+\langle r^*\rangle $ is $S$-free by Lemma~\ref{lemma:lineality}, and so is $P+\langle r^*\rangle $, a contradiction.

Thus  $P(\lambda^*)$  contains a point of $S$ in its interior. That is, there exists a point $\bar x\in S$ such that $[\lambda^* d_1+(1-\lambda^*)d_2](\bar x-f)<1$ and $d_i(\bar x-f)<1$ for $i=2,\ldots,t$. Since $P$ is $S$-free, $d_1(\bar x-f)>1$. Thus there exists $\bar\lambda>\lambda^*$ such that $[\bar \lambda d_1+(1-\bar \lambda)d_2](\bar x-f)=1$. Note that, since $P(\bar \lambda)\cap \conv(S)$ is bounded, there is a finite number of points of $S$ in the interior of $P(\bar\lambda)$. So we may choose $\bar x$ such that $\bar\lambda$ is maximum. Thus $P(\bar \lambda)$ is $S$-free and $\bar x$ is in the relative interior of the facet of $P(\bar \lambda)$ defined by $[\bar \lambda d_1+(1-\bar \lambda)d_2](x-f)\leq 1$.

Note that, for $i=2,\ldots, t$, if $d_i(x-f)\leq 1$ defines a facet of $P$ with a point of $S$ in its relative interior, then it also defines a facet of $P(\bar\lambda)$ with a point of $S$ in its relative interior, because $P\subset P(\bar\lambda)$. Thus repeating the above construction at most $t$ times, we obtain a set $B$ satisfying the lemma.
\end{proof}

\begin{remark}\label{rmk:not-full} Let $C$ and $K$ be as in Lemma~\ref{lemma:max-S-free}. Given any maximal $S$-free convex set $B=\{x\in\R^n\st a_i(x-f)\leq 1,\,\, i=1,\ldots,k\}$ containing $C$, then $a_1,\ldots,a_k\in K^*$. If $\rec(C)$ is not full dimensional, then the origin is not an extreme point of $K^*$. Since all extreme points of $K^*$ are contained in $\{0\}\cup\hat K$, in this case $\cl(\conv(\hat K))=K^*$. Therefore, when $\rec(C)$ is not full dimensional, every maximal $S$-free convex set containing $C$ satisfies the statement of  Lemma~\ref{lemma:max-S-free}.
\end{remark}

\bigskip

\begin{proof}[Proof of Theorem~\ref{thm:min-ineq-intr}]$\,$

\noindent We first show that any valid function is dominated by a function of the form $\psi_B$, for some maximal $S$-free convex set $B$ containing $f$ in its interior.

Let $\psi$ be a  valid function. By Lemma~\ref{lemma:sublinear}, we may
assume that $\psi$ is sublinear.
Let $K=\{r\in \R^n\st \psi(r)\leq 1\}$, and let $\hat K$ be defined
as in~(\ref{eq:set-T}). Note that $K=B_\psi-f$.  Thus, by
Remark~\ref{rmk:valid}, $\sum_{r\in \R^n}\rho_K(r)s_r\geq 1$ is
valid for $R_{f,S}$. Since $\psi$ is sublinear, it follows from
Theorem~\ref{thm:sublinear} that $\rho_K(r)\leq \psi(r)$ for every
$r\in \R^n$.

By Lemma~\ref{lemma:max-S-free}, there exists a maximal $S$-free
convex set $B=\{x\in\R^n\st a_i(x-f)\leq 1,\,\, i=1,\ldots,k\}$ such
that $a_i\in\cl(\conv(\hat K))$ for $i=1,\ldots,k$.

Then
$$\psi(r)\geq \rho_K(r)=\sup_{y\in\hat K}yr=\max_{y\in\cl(\conv(\hat K))} yr\geq \max_{i=1,\ldots,k}a_ir=\psi_B(r).$$
This shows that $\psi_B$ dominates $\psi$ for all $r\in\R^n$.
\bigskip

To complete the proof of the theorem, we need to show that, given a maximal $S$-free convex set $B$, the function $\psi_B$ is minimal. Consider any valid function $\psi$  dominating $\psi_B$. Then $B_\psi\supseteq B$ and $B_\psi$ is $S$-free. By maximality of $B$, $B=B_\psi$. By Theorem~\ref{thm:sublinear} and Remark~\ref{rmk:K-poly}, $\psi_B(r)\leq\psi(r)$ for all $r\in \R^n$, proving $\psi=\psi_B$.

\end{proof}

\begin{figure}[htbp]\label{fig:facet-tilting}
\centering
\includegraphics[height=3.0in]{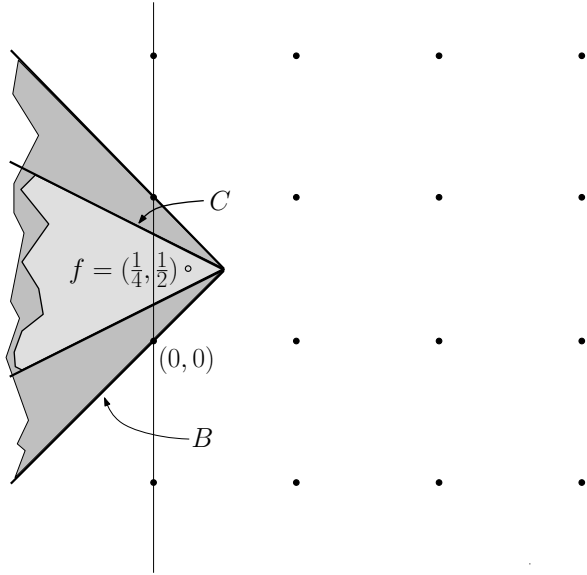}
\caption{Illustration for Example~\ref{example:ex1}}
\end{figure}

\begin{example}\label{example:ex1}
We illustrate the ideas behind the proof in the following
two-dimensional example. Consider $f = (\frac{1}{4},\frac{1}{2})$
and $S =\{(x_1, x_2) \st x_1 \geq 0)\}$. See
Figure~\ref{fig:facet-tilting}. Then the function $\psi(r) =
\max\{4r_1+8r_2, 4r_1 - 8r_2\}$ is a valid linear inequality for
$R_{f,S}$. The corresponding $B_\psi$ is $\{(x_1, x_2) \st 4(x_1 -
\frac{1}{4}) +8(x_2 - \frac{1}{2}) \leq 1, 4(x_1-\frac{1}{4})-8(x_2
- \frac{1}{2})\leq 1\}$. Note that $B_\psi$ is not a maximal
$S$-free convex set and it corresponds to $C$ in
Lemma~\ref{lemma:max-S-free}. Following the procedure outlined in
the proof, we obtain the maximal $S$-free convex set $B = \{(x_1,
x_2) \st 4(x_1 - \frac{1}{4}) + 4(x_2 - \frac{1}{2}) \leq 1,
4(x_1-\frac{1}{4})- 4(x_2 - \frac{1}{2})\leq 1\}$. Then, $\psi_B(r)
= \max\{4r_1+4r_2,4r_1 - 4r_2\}$ and $\psi_B$ dominates $\psi$.
\end{example}

\begin{remark}\label{rmk:psi>0} Note that $\psi$ is nonnegative if and only if $\rec(B_\psi)$ is not full-dimensional. It follows from Remark~\ref{rmk:not-full} that, for every maximal $S$-free convex set $B$ containing $B_\psi$, we have $\psi_B(r)\leq\psi(r)$ for every $r\in\R^n$ when $\psi$ is nonnegative.
\end{remark}

A statement similar to the one of Theorem~\ref{thm:min-ineq-intr} was shown by  Borozan-Cornu\'ejols~\cite{BorCor} for a model similar to~\eqref{eq:RfS} when $S=\Z^n$ and the vectors $s$ are elements of $\R^{\Q^n}$. In this case, it is easy to show that, for every valid inequality $\sum_{r\in\Q^n}\psi(r)s_r\geq 1$, the function  $\psi:\,\Q^n\rightarrow\R$ is nonnegative. Remark~\ref{rmk:psi>0} explains why in this context it is much easier to prove that minimal inequalities arise from maximal lattice-free convex sets.

\end{document}